\documentclass[12pt]{amsart}

\usepackage{amsmath,amsthm,amssymb}
\usepackage{amscd}
\usepackage{amsthm,amsfonts,amsmath}
\usepackage{amsmath, amsfonts, amsthm , amssymb}
\usepackage{amscd, times}
\usepackage[all]{xy}
\usepackage{enumerate,array}

\newtheorem{thm}{Theorem}
\newtheorem{prop}{Proposition}

\theoremstyle{remark}
\newtheorem{rem}{Remark}

\theoremstyle{definition}

\newtheorem{ex}{Example}

\newcommand{\Hom}{\operatorname{Hom}}

\newcommand{\AAA}{\mathcal{ A}}

\newcommand{\GG}{\mathcal{ G}}
\newcommand{\LLL}{\mathcal{ L}}
\newcommand{\BB}{\mathcal{ B}}

\newcommand{\Q}{\mathbb{ Q}}
\newcommand{\R}{\mathbb{ R}}
\newcommand{\LA}{\mathbb{ L}}

\newcommand{\rk}{\operatorname{rk}}

\newcommand{\Z}{\mathbb{ Z}}

\DeclareMathOperator{\dg}{deg}

\begin{document}
\title[Gauge gropus and spaces of connections] {The rational homology ring  of the based loop space of the  gauge groups and the  spaces of connections on a four-manifold}
\author{Svjetlana Terzi\'c}
\address{Faculty of Natural Sciences and Mathematics, University of Montenegro, D\v zord\v za Va\v singtona bb,
         81000 Podgorica, Montenegro}
\email{sterzic@ac.me}
\date{}
%\subjclass[2000]{Primary 57T20, 55P620; Secondary 55P35, 57T35.?????}

\begin{abstract}

We provide the rational-homotopic proof that the ranks of the homotopy groups of a simply connected four-manifold depend only on its second Betti number. We also consider the based loop spaces of  the gauge groups and the spaces of connections of 
a simply connected  four-manifold and, appealing to~\cite{TI} and using the models from  rational homotopy theory, we obtain the explicit formulas for their rational Pontrjagin homology rings.   
\end{abstract}

 \maketitle

\section{Introduction}
Rational homotopy theory can be seen as the  torsion free part of the  homotopy theory.  In the 1960's Sullivan proved that  simply connected topological spaces and continuous maps between them can be rationalized to  
topological spaces $X_{\Q}$ and  maps $f_{\Q} : X_{\Q}\to Y_{\Q}$ such that  $H_{k}(X_{\Q}) = H_{k}(X, \Q)$ and $\pi _{k}(X_{\Q})=\pi _{k}(X)\otimes \Q$.  The rational homotopy type of $X$ is defined to be the homotopy type of $X_{\Q}$ and the rational homotopy class of $f : X\to Y$ is defined to be  the homotopy class of $f_{\Q} : X_{\Q}\to Y_{\Q}$. Rational homotopy theory studies those  properties of spaces and maps between them which depend only on the rational homotopy type of a space and rational homotopy class of a map.

The main advantage of  rational homotopy theory is its great computability. This advantage is due to its explicit algebraic modeling obtained by Quillen~\cite{Q} and Sullivan~\cite{S}. As for the topological spaces the result is that the rational homotopy type of a topological space is determined by the isomorphism class of its algebraic model.

In this paper we continue our work on rational homotopy of simply connected four-manifolds, their gauge groups and spaces of connections. In that direction we  prove, only by the means of rational homotopy theory, that the rational homotopy groups of a simply connected four-manifold can be expressed only in terms of its second Betti number.  We also consider the based loop spaces  of the  gauge groups and the spaces of connection on a simply connected four manifold. Using the results from~\cite{TI} on the  rational cohomology structure of the gauge groups and the spaces of connection, we provide the explicit formulas for the rational Pontrjagin homology rings of the corresponding based loop spaces.

We refer to~\cite{FHT} for very detailed and comprehensive background on rational homotopy theory and its algebraic modeling.    
  
\section{On homotopy groups of simply connected four manifolds}

\subsection{Review of the known results}
Let  $M$ be a closed simply connected four-manifold. We recall the notion of the intersection form for $M$. The symmetric bilinear form
\[
Q_{M} : H^{2}(M, \Z) \times H^{2}(M, \Z) \to \Z,
\]
defined by
\[
Q_{M}(x,y) = \left\langle x\cup y, \left[ M\right]\right\rangle ,
\]
where $\left[ M\right]$ is the fundamental class for $M$, is called the intersection form for $M$. 

It is well known the result of Pontryagin-Wall~\cite{M} that the homotopy type of a simply connected four-manifold is classified by its intersection form. 

The intersection form $Q_{M}$ can be diagonalised over $\R$ with $\pm 1$ as the diagonal elements. Following the standard notation,  $b_{2}^{+}(M)$ denotes the number of $(+1)$ and $b_{2}^{-}(M)$ the number of $(-1)$ in the diagonal form for $Q_{M}$. The numbers  $b_{2}(M) = b_{2}^{+}(M)+b_{2}^{-}(M)$ and $\sigma (M) = b_{2}^{+}(M)-b_{2}^{-}(M)$  are known as the rank and the signature of $M$. 

 The rational homotopy type of a simply connected four-manifold is, by the result of~\cite{T}, classified by its rank and signature. 

The real cohomology algebra of $M$ for $b_{2}(M)\geq 2$  can be easily described using the  intersection form for $M$ as follows:
\begin{equation}\label{cohom}
H^{*}(M) \cong \R [x_1,\ldots x_{b_{2}^{+}},x_{b_{2}^{+}+1},\ldots ,x_{b_{2}}]/I,
\end{equation}
where $\deg x_i=2$, and the ideal $I$  is generated by the equalities 
\[
x_1^2=\cdots =x_{b_{2}^{+}}^2 = -x_{b_{2}^{+}+1}=\cdots =-x_{b_2^2},
\]
\[
x_ix_j=0,\;\; i\neq j.
\]
Any simply connected four-manifold $M$ is formal in the sense of rational homotopy theory.  The following explicit formulae  for the ranks of the third and the  fourth homotopy group  of $M$  were obtained in~\cite{T} by applying  Sullivan minimal model theory to  the  cohomology algebra of $M$:
\[
\rk \pi _{2}(M) = b_{2}(M),\;\; \rk \pi _{3}(M) = \frac{b_{2}(b_{2}+1)}{2}-1,\;\;
 \rk \pi _{4}(M) = \frac{b_{2}(b_{2}^2-4)}{3}.
 \]
The calculation procedure of the ranks of homotopy groups based on the Sullivan minimal theory carried out in~\cite{T}     suggested that the ranks of all homotopy groups of $M$ depend only on  the second Betti number of $M$, but we were not able to prove it explicitly using that method.  

A  formula for all   homotopy groups of a simply connected four manifold is   obtained  later in~\cite{DL} by proving that any simply connected four-manifold admits $S^1$-covering of the form $\# _{b_{2}(M)-1}S^{2}\times S^{3}$. It implies that
\[
\pi _{k}(M) = \left\{
\begin{array}{c}
\oplus _{b_{2}(M)}\Z \;\; \text{for}\;\; k=2,\\
\pi _{k}(\# _{b_{2}(M)-1}S^{2}\times S^{3}),\;\; \text{for}\;\; k>2
\end{array}
\right .
\]
This formula in particular proves  that the ranks of the homotopy groups of $M$ depend only on the second Betti number $b_{2}(M)$.  The recent results obtained in this direction are much stronger. Denote by $\Omega M$ the space of based loops on $M$. It is proved  in~\cite{Stiv} that for a simply connected four-manifolds $M$ and $N$ it holds:
\[
\Omega M\cong \Omega N \Longleftrightarrow H^{2}(M)\cong H^{2}(N).
\] 
This is done via  homotopy theoretic methods by decomposing $\Omega M$ into a product of spaces, up to homotopy:
\[
\Omega M\cong S^{1}\times \Omega (S^2\times S^3)\times \Omega (J\vee (J\wedge \Omega (S^2\times S^3))),
\]
where $J = {\bf \vee}_{i=1}^{k-2}(S^2\vee S^3)$ if $k>2$ and $J= \star$ if $k=2$.   
Since $H^{2}(M) = \pi _{2}(M) = \oplus _{b_{2}(M)}\Z $ we have  that the based loop spaces $\Omega M$ and $\Omega N$ are homotopy equivalent if and only if $b_{2}(M) = b_{2}(N)$. On the other hand, the application  of  the Sullivan minimal model theory to the  algebra~\eqref{cohom} gives that, in general, the ranks of homotopy groups of a simply connected four-manifold $M$  can be expressed only in terms of the  rank and possible signature of the  corresponding intersection form. This  together proves that the ranks of homotopy groups for $M$  depend only on $b_{2}(M)$.

We would like to mention here the results from~\cite{BAB} which give the explicit formulas on the ranks of   homotopy groups of the suspension of a finite simplicial complex:
\[
\rk \pi _{j+1}(\sum X) = \frac{(-1)^{j}}{j}\sum _{d|j} (-1)^{d}\mu (\frac{i}{d}) S_{d}(2-P(z)),
\]
where $P(z)$ is the Poincar\' e polynomial for $X$, by $\mu$ is denoted the  M\"obius function and $S_{d}$ is the Newton polynomial of the roots of the polynomial inverse to the polynomial  $2-P(z)$.

 In the case of a simply connected four-manifold $M$, the following explicit formula for the rational homotopy groups has been obtained recently   in~\cite{SS}:
\[
\rk (\pi _{n+1}(M)) = \sum _{d|n}(-1)^{n+\frac{n}{d}}\frac{\mu (d)}{d}\sum _{a+2b=\frac{n}{d}}(-1)^{b}{a+b \choose b}\frac{b_{2}^{a}}{a+b},
\]  
where $b_{2}$ denotes the second Betti number of $M$.

\subsection{The background from rational homotopy theory}
We provide here  the simple  proof, based only on the methods of rational homotopy theory,  that the rational homotopy groups of a simply connected four-manifold $M$ can be expressed only  in terms of  $b_{2}(M)$. In this way we complete the work~\cite{T} by the means of the  rational homotopy theory.

Let us recall some notions from  rational homotopy theory, for the background we refer to~\cite{FHT}. Let $TV$ denotes the tensor algebra on a graded vector spaces $V$. It is a graded  Lie algebra with the commutator bracket. The sub Lie algebra generated by $V$ is called the free graded Lie algebra on $V$ and it is denoted by $\LA _{V}$.  An element in $\LA _{V}$ is said to have the  bracket length $k$ if it is linear combination of the elements of the form $[v_1,\ldots, [v_{k-1},v_{k}],\ldots]$. We obtain the decomposition $\LA _{V} = \oplus \LA _{V}^{(i)}$, where $\LA _{V}^{(i)}$ consists of elements having the bracket length $i$.  Then it holds $TV \cong U\LA _{V}$,  where $U\LA _{V}$ is the universal enveloping algebra for $\LA _{V}$.  Let $(\LA _{V}, d)$ be a differential graded free Lie algebra.  The linear part of $d$ is the differential $d_{V} : V\to V$ defined by $dv - d_{V}v \in \oplus _{k\geq 2}\LA _{V}^{(k)}$. An algebra $(\LA _{V}, d)$ is said to be minimal if the linear part $d_{V}$ of its differential vanishes.

The Quillen functor $C_{*}$  is a functor from the connected chain Lie algebras $\{(L; d_{L})\}$  to one-connected cocommutative chain coalgebras $\{ C_{*}(L; d_{L})\}$. By  dualizing Quillen's  construction  one obtains commutative differential graded algebras $\{(C^{*}(L; d_{L})\}$. In the case when  $(L; d_{L})$ is a chain algebra, $C^{*}(L; d_{L})$ is a cochain algebra.  

Let $X$ be a simply connected space  with rational homology of finite type and $A_{PL}(X)$ its commutative cochain algebra. A Lie model for $X$ is a connected chain Lie algebra $(L; d_{L})$ of finite type with a differential graded algebras quasi-isomorphism
\[
m : C^{*}(L; d_{L}) \stackrel{\cong}{\rightarrow} A_{PL}(X).
\] 
Any such space $X$ has a minimal free Lie model which is unique, up to isomorphism. For  a minimal free model $(\LA _{V}, d)$   for $X$, it is known  that there exists a Lie algebra isomorphism 
\begin{equation}\label{isom}
H(\LA _{V}, d) \stackrel{\cong}{\rightarrow} \pi _{*}(\Omega X)\otimes _{\Z} \Q ,
\end{equation}
 where the Lie algebra structure on $\pi _{*}(\Omega X)\otimes _{\Z} \Q$ is given by the Whitehead product. This isomorphism is dual to the isomorphism $H(\wedge V_{X})\stackrel{\cong}{\rightarrow} H^{*}(X, \Q)$ in the context of Sullivan algebras, where $\wedge V_{X}$ denotes the minimal model for $X$. The  isomorphism~\eqref{isom} suspends to 
\begin{equation}\label{isom2}
sH(\LA _{V}, d)\stackrel{\cong}{\rightarrow} \pi _{*}(X)\otimes _{\Z }\Q. 
\end{equation}

Assume now that $X$ is an adjunction space, meaning that
\[
X = Y \cup _{f}(\cup _{\alpha}D^{n_{\alpha}+1}), \; \text{where}:
\]
\begin{itemize}
\item $Y$ is simply-connected space with rational homology of finite type,
\item $f = \{ f_{\alpha} : (S^{n_{\alpha}}, \ast )\to (Y, y_{0})\}$,
\item the cells $D^{n_{\alpha}+1}$ are all of dimension $\geq 2$, with finitely many of them  in any dimension.
\end{itemize}
Let $\tau : sH(\LA _{V},d)\stackrel{\cong}{\rightarrow} \pi _{*}(X)\otimes \Q$ be an isomorphism, where $(\LA _{V}, d)$  is  a minimal free Lee model for $X$. Using  this isomorphism the classes $[f_{\alpha}]\in \pi _{n_{\alpha}}(X)$ determine the classes $s[z_{\alpha}] = \tau ^{-1}([f_{\alpha}])$ represented by the cycles $z_{\alpha}\in \LA _{V}$. Let $W$ be a graded vector space with basis $\{w_{\alpha}\}$ such that $\deg w_{\alpha} = n_{\alpha}$. One can extend $\LA _{V}$ to $\LA _{V\oplus W}$ by putting  $dw_{\alpha} = z_{\alpha}$. It is proved in~\cite{FHT} that the chain algebra $(\LA _{V\oplus W}, d)$ is a Lie model for $X$.

\subsection{On rational homotopy groups of simply connected four-manifolds} In this section, appealing to  the methods of  rational homotopy theory, we prove:

\begin{prop}
The rational homotopy groups of a simply-connected four manifold depend only on its second Betti number.
\end{prop} 

\begin{proof}
It is the  classical result, see ~\cite{M1} and~\cite{W},  that a simply-connected four-manifold $M$ is homotopically an adjunction space:
\begin{equation}\label{cof}
M \cong (\vee _{b_{2}(M)}S^{2})\cup _{f} D^{4}.
\end{equation}
A Lie model for the space $X = \vee _{\alpha}S^{n_{\alpha}+1} = pt \cup _{f}(\cup _{\alpha}D^{n_{\alpha}+1})$ is (see~\cite{FHT}) given by $(\LA _{V}, 0)$, where $V=\{V_{i}\}_{i\geq 1}$ is a vector space of finite type with basis $v_{\alpha}$ such that $\deg v_{\alpha} = n_{\alpha}$. It implies that a Lie model for $\vee _{b_{2}(M)}S^{2}$ is given by
$(\LA _{V},0)$, where $V$ has a basis $v_1,\ldots ,v_{b_{2}(M)}$ such that $\deg v_{\alpha}=1$ for all $1\leq \alpha \leq b_{2}(M)$. For  $[f]\in \pi _{3}(\vee _{b_{2}(M)}S^{2}, x_{0})$ denote by $z$, according to the previous section,  a cycle in $\LA _{V}$  such that $s[z]=\tau ^{-1}([f])\in sH(\LA _{V})$. We have that $\deg z= 2$. Let $W = \LLL (w)$  be a vector space with the basis $w$ such that $\deg w=3$.  

Then the Lie model for $M$ is given by 
\[
(\LA _{V\oplus W} = \LA _{\LLL(v_1,\ldots ,v_{b_2},w)}, d),\; \text{where}\; dv_i = 0\; dw = z \in \LA _{V}.
\]
We note that  the ranks of the cohomology groups in $H(\LA _{\oplus W}, d)$  depend  only on $b_{2}(M)$. Together with~\eqref{isom} this proves the statement.
\end{proof}

\section{On gauge groups and spaces of connections} 
Let $\pi : P\to M$ be a $G$-principal bundle, where $G$ is a compact semisimple simply connected Lie group and $M$ is a compact simply connected four-manifold. The gauge group  $\GG$ of this principal bundle is the group of $G$-equivariant
automorphisms of $P$  which induce identity on the base. Let $\AAA$ denotes the space of all connections and $\AAA ^{*}$ the space of all irreducible connections on the bundle $P$. It is assumed that these spaces are equipped with certain Sobolev topologies $L_{p-1}^{2}$ for $\AAA$ and $L_{p}^{2}$ for $\GG$, where $p$ is large enough. 

The action of $\GG$ on $\AAA$ and $\AAA ^{*}$ is not free in general. Instead one should consider  the group $\GG _{0}$ of the gauge transformations  for $P$ which fix one fiber. The group $\GG _{0}$  acts freely on the spaces of connections and irreducible connections. The group $\tilde{\GG} = \GG/Z(G)$ also acts  freely on $\AAA ^{*}$, where $Z(G)$ is the center of the group $G$, which gives a fibration
\begin{equation}\label{fib}
\tilde{\BB} = \AAA /\GG _{0},\;\; \tilde{\BB}^{*} = \AAA ^{*}/\GG _{0},\;\; \BB ^{*} =\AAA ^{*}/\tilde{\GG}.
\end{equation}
It is known, see~\cite{D} that $\AAA$ is contractible, then $\AAA ^{*}$ is weakly homotopy equivalent to $\AAA$ and $\tilde{\BB}^{*}$ is weakly equivalent to $\tilde{\BB}$.

The well known result of Singer~\cite{SI} states that the weak homotopy type of  $\GG _{0}$ is independent of $P$:
\begin{equation}\label{sin}
\GG _{0}\approx Map _{*}(M, BG).
\end{equation}
This means  that $\GG _{0}$ is weakly homotopy equivalent to the pointed mapping space $Map_{*}(M, BG)$.

Let $\GG ^{e}$ be the identity component of the gauge group $\GG$. The rational cohomology algebras  of the spaces $\GG ^{e}$, $\tilde{\BB}$ and   $\BB ^{*}$ are computed in~\cite{TI} and they are as follows:
\begin{itemize}
\item $H^{*}(\GG ^{e})$ is an exterior algebra in $(b_{2}(M)+2)\rk G-1$  generators of odd degree, where the number of generators in degree $j$ is equal to $b_{2}(M)\rk \pi _{j+2}(G)+\rk \pi _{j}(G)+\rk \pi _{j+4}(G)$.
\item $H^{*}(\tilde{\BB})$ is a polynomial algebra in $(b_{2}(M)+1)\rk G-1$ generators of even degree, where the number of generators of degree $j$ is equal to $b_{2}(M)\rk \pi _{j+1}(G) + \rk \pi _{j+3}(G)$.
\item $H^{*}(\BB ^{*})$ is a polynomial algebra in $(b_{2}(M)+2)\rk G-1$ generators of even degree, where the number of generators in degree $j$ is equal to $b_{2}(M)\rk \pi _{j+1}(G) + \rk \pi _{j-3}(G) + \rk \pi _{j+3}(G)$.
\end{itemize}
 
\begin{rem}\label{minmod} Since for all these  spaces their rational cohomology algebras are free,  they are all formal in the sense of  rational homotopy theory as remarked in~\cite{TI}. Therefore the minimal models, in the sense of rational homotopy theory, for the spaces $\GG ^{e}$, $\tilde{\BB}$ and   $\BB ^{*}$ coincide with the minimal models for their rational cohomology algebras. Furthermore, being free, their rational cohomology algebras coincide with their minimal models.  
\end{rem}

\subsection{The rational Pontrjagin homology ring of a based loop space}
For a topological spaces $X$ its based loop space $\Omega X$ is an $H$-space where one of possible multiplication of the loops is given by loop concatenation. The ring structure induced in $H_{*}(\Omega X)$ is called Pontrjagin homology ring. We are interested in Pontrjagin homology ring of the spaces $\GG$, $\tilde{\BB }$ and $\BB ^{*}$.
  
Let us   recall the theorem of Milnor and Moore and some constructions from the  rational homotopy theory which will be directly  applied  to describe the  Pontrjagin homology rings of the spaces we consider. The theorem of Milnor-Moore~\cite{MM}  states that the rational homology algebra of the based loop space $\Omega X$ of a simply connected space $X$ is given by
\begin{equation}\label{MM}
H_{*}(\Omega X)\cong UL _{X} \cong T(L _{X})/\langle xy-(-1)^{\deg x\deg y}yx-[x,y]\rangle,
\end{equation}
where $L_{X}$ is  the rational homotopy Lie algebra of $X$ and $UL _{X}$ is the universal enveloping algebra for $L_{X}$. The algebra $L_{X}$ is a graded Lie algebra defined by  $L_X=(\pi_*(\Omega X)\otimes \Q)$ and the commutator $[\, , \,]$ is given by the Samelson product. 

It is the result of  rational homotopy theory~\cite{FHT} that there is an isomorphism between the rational homotopy Lie algebra
$L_X$ and the homotopy Lie algebra $L$ of the minimal model for $X$. The homotopy Lie algebra $L$   is defined as follows.  Let $(\wedge V, d)$ be a minimal model for $X$. Then $sL=\Hom(V, \Q)$, where the suspension $sL$ is    defined in the standard way by $(sL)_i=(L)_{i-1}$. The Lie brackets  $ L$ are defined by
\begin{equation}
\label{bracket}
 \langle v; s[x,y]\rangle= (-1)^{\deg y+1}\langle
 d_1v;sx,sy\rangle \quad \text{ for } x,y\in L, v\in V.
 \end{equation}
Here $d_1$ is quadratic part in the differential $d$ for the minimal model $(\wedge V, d)$ and it is defined  by
$d-d_1 \in \wedge ^{k\geq 3}V$. Thus $d_1v = v_1\wedge v_2$ for some $v_1,v_2\in V$ and the expression  $\langle
 d_1v;sx,sy\rangle$ is defined by
\[
\langle v_1\wedge v_2; sx, sy\rangle=\langle
v_{1}; sx\rangle\langle v_{2};sy\rangle - \langle
v_{2}; sx\rangle\langle v_{1};sy\rangle
\]

In order to apply Milnor-Moore theorem we will assume that the spaces $ \tilde{\BB }$ and $ \BB ^{*}$ are simply connected. Using~\eqref{fib} we see that  that  this condition is equivalent to the condition that the groups $\GG _{0}$ and $\tilde{\GG}$ are connected, which  is further equivalent to the condition that the gauge group $\GG$ is connected.

\begin{thm}\label{homol} Assume that the gauge group $\GG$ is connected. Then the rational Pontrjagin homology  rings   of the based loop spaces $\Omega \tilde{\BB}$ and $\Omega \BB ^{*}$ are as follows:
\begin{itemize}
%\item $H_{*}(\Omega \GG ^{e})$ is a polynomial algebra in $(b_{2}(M)+2)\rk G-1$ even degree generators, where the number of generators in degree $j$ is  $b_{2}(M)\rk \pi _{j+3}(G)+\rk \pi _{j+1}(G)+\rk \pi _{j+5}(G)$.
\item $H_{*}(\Omega \tilde{\BB})$ is an exterior algebra in $(b_{2}(M)+1)\rk G-1$ generators of odd degree, where the number of generators of degree $j$ is equal to $b_{2}(M)\rk \pi _{j+2}(G) + \rk \pi _{j+4}(G)$.
\item $H_{*}(\Omega \BB ^{*})$ is an exterior algebra in $(b_{2}(M)+2)\rk G-1$ generators of odd degree, where the number of generators in degree $j$ is equal to $b_{2}(M)\rk \pi _{j+2}(G) + \rk \pi _{j}(G) + \rk \pi _{j+4}(G)$.
\end{itemize}
 \end{thm}

\begin{proof}
By Remark~\ref{minmod} we have that the differentials  in the minimal models for the spaces $\tilde {\BB}$ and $\BB ^{*}$  are trivial implying that    their quadratic parts  $d_1$  are  trivial as well. The definition of the brackets given by~\eqref{bracket} implies  that the corresponding homotopy Lie algebras are commutative. It implies,  by~\eqref{MM} that $UL$ is commutative and, therefore  the loop space homology algebra for these spaces coincide with their rational cohomology algebras with the degrees of generators shifted by $-1$. 
\end{proof}

Let us comment on the assumption that the gauge group $\GG$ is connected. The group $\GG$ is not connected in general
The result of Singer~\eqref{sin} together with the homotopy cofibration~\eqref{cof} gives the exact homotopy sequence, see~\cite{TI}:
\begin{equation}
 \pi _{3}(G)\leftarrow \oplus _{b_{2}(M)}\pi _{2}(G)\leftarrow [M; BG]\leftarrow \pi _{4}(G)\leftarrow \ldots .
\end{equation}
It implies that  $\GG _{0}$ is connected if $\pi _{2}(G) = \pi _{4}(G) =0$. This  condition is satisfied  when
 $G=SU(n)$ for $n\geq 3$ and $G=Spin (n)$ for $n\geq 6$ (this follows from the Bott periodicity).

 \begin{ex}
 Let us consider the  $SU(3)$-principal bundles over a simply connected  four-manifold $M$.  The gauge group of any such bundle is connected. It is well known that $\rk \pi _{3}(SU(3)) = \rk \pi _{5}(SU(3)) = 1$ and $\rk \pi _{j}(SU(3))=0$ for $j\neq 3,5$.  Therefore, Theorem~\ref{homol} implies:
 
 \[
 H_{*}(\Omega \tilde{\BB }) \cong \wedge (x_1,\ldots ,x_{b_{2}(M)+1}, y_{},\ldots ,y_{b_{2}(M)}),\; \dg x_i = 1,\; \dg y_{j}=3.
 \]
 \[
 H_{*}(\Omega \BB ^{*}) \cong \wedge  (x_1,\ldots ,x_{b_{2}(M)+1},y_1,\ldots ,y_{b_{2}(M)+1}, z),
 \]
 \[ \dg x_i =1,\; \dg y_{j}=3,\; \dg z=5.
 \]
 \end{ex} 
 \begin{ex}
 Let us now consider the  $SU(2)$-principal bundles over a simply connected  four-manifold $M$.  It is well known that  $\pi _{4}(SU(2))=\Z _{2}$. In this case the  fundamental  group  $\pi _{1}(\BB ^{*})$ depends on the  $SU(2)$-principal bundle $P\to M$.  This fundamental group is   computed in~\cite{O} and  we recall these results.   For $G=SU(2)$ we have that $\pi _{0}(\GG ) = [M, S^3]$ and due to Steenrod theorem it follows:
 \[
 \pi _{0}(\GG ) = \left\{
 \begin{array}{cc}
 0, & \text{if the intersection form for $M$ is odd}\\
 \Z _{2}, & \text{if the intersection form for $M$ is even}.  
 \end{array}
 \right .
\]  
It implies, using~\eqref{fib} that
\[
 \pi _{1}(\tilde{\BB} ) = \left\{
 \begin{array}{cc}
 0, & \text{if the intersection form for $M$ is odd}\\
 \Z _{2}, & \text{if the intersection form for $M$ is even}.  
 \end{array}
 \right .
\]
\begin{itemize}
\item Assume that  the intersection form for $M$ is odd.  Theorem~\ref{homol} gives
 \[
 H_{*}(\Omega \tilde{\BB}) \cong \wedge  (x_1,\ldots ,x_{b_{2}(M)}),\;\; \dg x_i =1.
 \]
  
Also the fibration $Z(G)\to \GG \to \tilde{\GG}$ gives the exact sequence
\begin{equation}\label{seq}
\ldots \rightarrow \Z _{2}\stackrel{j_{*}}{\rightarrow} \pi _{0}(\GG )\to \pi _{0}(\tilde{\GG})\to 0,
\end{equation}
which  implies that $\tilde {\GG}$ is connected meaning that $\BB ^{*}$ is simply connected. Therefore in this case Theorem~\ref{homol} gives
\[
 H_{*}(\Omega \BB ^{*}) \cong \wedge  (x_1,\ldots ,x_{b_{2}(M)},y),\; \dg x_i =1,\; \dg y =3.
 \]
 \item  Assume that  the intersection form for $M$  is even.   It is proved in~\cite{O} that the map $j_{*}$  from~\eqref{seq} depends on the second Chern number of the principal $SU(2)$-bundle $P$ over $M$. More precisely,  it is proved  that $j_{*}$ is $0$-map when $c_{2}(P)$ is even and $j_{*}$ is surjective when $c_{2}(P)$ is odd. It implies that $\pi _{0}(\tilde{\GG}) = \Z _{2}$ if $c_{2}(P)$ is even and $\pi _{0}(\tilde{\GG }) = 0$ if $c_{2}(P)$ is odd. Thus   
 \[
 \pi _{1}(\BB ^{*}) = \left\{
 \begin{array}{cc}
 \Z_{2}, & \text{$c_{2}(P)$ is even}\\
 0,  &  \text{$c_{2}(P)$ is odd}.
 \end{array}
 \right .
 \]
 Therefore when the  intersection form for $M$ is even and the second Chern number  of the  $SU(2)$-principal bundle $P\to M$  is odd, Theorem~\ref{homol} gives that:
 \[
 H_{*}(\Omega \BB ^{*}) \cong \wedge  (x_1,\ldots ,x_{b_{2}(M)},y),\; \dg x_i =1,\; \dg y =3.
 \]
 \end{itemize}
\end{ex} 
\bibliographystyle{amsalpha}

\begin{thebibliography}{BBCG}
\bibitem{BAB} I.~K.~Babenko, {\sl On analytic properties of Poincare series of a loop space}, Mat.~Zametki {\bf 27}, No.~5, (1980), 751-765 (in Russian), Math.~Notes {\bf 27} (1980), 759-767 (English translation).
\bibitem{SS} Sa.~Basu and So.~Basu, {\sl Homotopy groups and periodic geodesics   of closed $4$-manifolds}, arXiv:1303.3328.
\bibitem{Stiv} Piotr Beben, Stephen Theriault, {\sl The loop space homotopy type of simply-connected four-manifolds and their generalization}, Advances in Mathematics {\bf 262}, (2014), 213-238.
\bibitem{D} S.~K.~Donaldson and P.~B.~Kronheimer, {\em The Geometry of Four-Manifolds}, Oxford University Press, 1990. 
\bibitem{DL} Haibao Duan, Chao Liang, {\sl Circle bundles over $4$-manifolds}, Arch.~Math.~(Basel) {\bf 85}, no.~3, (2005), 278-282.
\bibitem{FHT}
Yves F{\'e}lix, Stephen Halperin and Jean-Claude Thomas,
     {\em Rational homotopy theory},
    Graduate Texts in Mathematics \textbf{205},
 Springer-Verlag, 2001.
\bibitem{M} R.~Mandelbaum, {\sl Four-dimensional topology: an introduction}, Bull.~Amer.~Math.~Society {\bf 2}, (1980), 1-159.
\bibitem{M1} J.~Milnor, {\sl On simply connected 4-manifolds}, Symposium Internacional de Topologia Algebraica, (1958), 122-128.
\bibitem{MM} John Milnor and John Moore, {\sl On the structure of {H}opf algebras},
Annals of Mathematics, \textbf{81} (1965), 211--264.
\bibitem{O} Hiroshi Ohta, {\sl On the fundamental groups of moduli spaces of irreducible $SU(2)$-connections over closed $4$-manifolds}, Proc.~Japan Acad.~{\bf 66}, Ser.~A, no.~3, (1990), 89--92. 
\bibitem{Q}
D.~Quillen, {\em Rational homotopy theory}, Annals of Math.~{\bf 90} (1969), 205--295.
\bibitem{SI} I.~M.~Singer, {\sl Some remarks on the Gribov ambiguity}, Commun.~Math.~Phys.~{\bf 60}, (1978) 7--12.
\bibitem{S}
D.~Sullivan, {\em Infinitesimal computations in topology}, Publ.~I.~H.~E.~S~{\bf 47} (1977), 269--331.


\bibitem{T} S.~Terzi\'c, {\sl On rational homotopy of four-manifolds}, in Contemporary geometry and related topics, World.~Sci.~Publ., River Edge, NJ, (2004), 375-388.
\bibitem{TI} Svjetlana Terzi\'c, {\sl The rational topology of gauge groups and of spaces of connections}, Compositio Mathematicae, {\bf 141}, no.~1, (2005), 262-270.
\bibitem{W}J.~H.~C.~Whitehead, {\sl On simply connected 4-dimensional polyhedra}, Comm.~Math.~Helvet.,{\bf 22}, (1949), 49-92.

\end{thebibliography}

\end{document}